\documentclass[11pt]{amsart}%
\usepackage{amsfonts}
\usepackage{amsmath}
\usepackage{amssymb}
\usepackage{graphicx}
\usepackage{ulem}
\usepackage{xcolor}%
\providecommand{\U}[1]{\protect\rule{.1in}{.1in}}
\providecommand{\U}[1]{\protect\rule{.1in}{.1in}}
\numberwithin{equation}{section}
\newtheorem{thm}{Theorem}[section]

\newtheorem{prop}{Proposition}[section]

\theoremstyle{definition}

\begin{document}
\title[Hessian estimate]{Hessian estimate for semiconvex solutions to the sigma-2 equation}
\author{Ravi Shankar}
\author{Yu YUAN}
\address{University of Washington\\
Department of Mathematics, Box 354350\\
Seattle, WA 98195}
\email{shankarr@uw.edu, yuan@math.washington.edu}
\thanks{RS and YY are partially supported by
NSF Graduate Research Fellowship Program under grant No. DGE-1762114 and NSF grant DMS-1800495 respectively.}

\begin{abstract}
We derive a priori interior Hessian estimates for semiconvex solutions to the
sigma-2 equation. An elusive Jacobi inequality, a transformation rule under
the Legendre-Lewy transform, and a mean value inequality for the still
nonuniformly elliptic equation without area structure are the key to our
arguments. Previously, this result was known for almost convex solutions.

\end{abstract}
\date{\today}
\maketitle

\section{Introduction}

\label{sec:Intro}

In this paper, we prove a priori Hessian estimates for semiconvex solutions to
the quadratic Hessian equation
\begin{equation}
F\left(  D^{2}u\right)  =\sigma_{2}\left(  \lambda\right)  =\sum_{1\leq
i<j\leq n}\lambda_{i}\lambda_{j}=\frac{1}{2}\left[  \left(  \bigtriangleup
u\right)  ^{2}-\left\vert D^{2}u\right\vert ^{2}\right]  =1. \label{Esigma2}%
\end{equation}
Here $\lambda_{i}s$ are the eigenvalues of the Hessian $D^{2}u.$

\begin{thm}
\label{thm:Hess} Let $u$ be a smooth semiconvex solution to $\sigma_{2}\left(
D^{2}u\right)  =1$ on $B_{R}\left(  0\right)  \subset\mathbb{R}^{n}$ with
$D^{2}u\geq-K~I$ for any fixed $K>0.$ Then
\[
\left\vert D^{2}u\left(  0\right)  \right\vert \leq C\left(  n,K\right)
\exp\left[  C\left(  n,K\right)  \left\Vert Du\right\Vert _{L^{\infty}\left(
B_{R}\left(  0\right)  \right)  }^{2}/R^{2}\right]  .
\]

\end{thm}

Given the gradient bound in terms of $K$-convex function $u\left(  x\right)  $
(note that Trudinger's gradient estimates for $\sigma_{k}$ equations need no
semiconvexity of the solutions [T]), we can control $D^{2}u$ in terms of the
solution $u$ in $B_{2R}\left(  0\right)  $ as
\[
\left\vert D^{2}u\left(  0\right)  \right\vert \leq C\left(  n,K\right)
\exp\left[  C\left(  n,K\right)  \left\Vert u\right\Vert _{L^{\infty}\left(
B_{2R}\left(  0\right)  \right)  }^{2}/R^{4}\right]  .
\]
One quick application of the above estimates is a rigidity result for entire
semiconvex solutions with quadratic growth to (\ref{Esigma2}): every such
solution must be quadratic.

Recall any solution to the Laplace equation $\sigma_{1}\left(  D^{2}u\right)
=\bigtriangleup u=1$ enjoys a priori Hessian estimates; yet there are singular
solutions to the three dimensional Monge-Amp\`{e}re equation $\sigma
_{3}\left(  D^{2}u\right)  =\det D^{2}u=1$ by Pogorelov [P], which
automatically generalize to singular solutions to $\sigma_{k}\left(
D^{2}u\right)  =1$ with $k\geq3$ in higher dimensions $n\geq4.$

Sixty years ago, Heinze [H] achieved a Hessian bound for solutions to equation
$\sigma_{2}\left(  D^{2}u\right)  =1$ in dimension two by two dimensional
techniques. More than ten years ago, a Hessian bound for $\sigma_{2}\left(
D^{2}u\right)  =1$ in dimension three was obtained via the minimal surface
feature of the \textquotedblleft gradient\textquotedblright\ graph $\left(
x,Du\left(  x\right)  \right)  $ in the joint work with Warren [WY]. Along
this \textquotedblleft integral\textquotedblright\ way, Qiu [Q] has proved
Hessian estimates for solutions to the three dimensional quadratic Hessian
equation with $C^{1,1}$ variable right hand side. Hessian estimates for convex
solutions to general quadratic Hessian equations have also been obtained via a
new pointwise approach by Guan and Qiu [GQ]. Hessian estimates for almost
convex solutions to (\ref{Esigma2}) have been derived by a compactness
argument in [MSY]. Hessian estimates for solutions to Monge-Amp\`{e}re
equation $\sigma_{n}\left(  D^{2}u\right)  =\det D^{2}u=1$ and Hessian
equations $\sigma_{k}\left(  D^{2}u\right)  =1$ ($k\geq2$) in terms of the
reciprocal of the difference between solutions and their boundary values, were
derived by Pogorelov [P] and Chou-Wang [CW], respectively, using Pogorelov's
pointwise technique. Lastly, we also mention Hessian estimates for solutions
to $\sigma_{k}$ as well as $\sigma_{k}/\sigma_{n}$ equations in terms of
certain integrals of the Hessian by Urbas [U1,U2], Bao-Chen-Guan-Ji [BCGJ].

Note that the almost convexity condition for solutions in [GQ, (15)] and [MSY,
Theorem 1.1] is essential in both the respective arguments toward Hessian
estimates for quadratic Hessian equations. The mean value inequality in
\cite{WY} used the area structure of the equation. For semiconvex solutions,
an elusive Jacobi inequality, a transformation rule under Legendre-Lewy
transform, and a mean value inequality corresponding to the still nonuniformly
elliptic linearized operator without area structure are essential in our proof
of Theorem 1.1.

The bulk of Section \ref{sec:Prelims} is devoted to establishing the Jacobi
inequality, Proposition \ref{prop:Jac}, $\sum F_{ij}b_{ij}\geq\sum F_{ij}%
b_{i}b_{j}$ with $F_{ij}$ the linearized operator, $b=\frac{1}{4}\ln
\lambda_{\max}\left(  D^{2}u\right)  ,$ and $u\left(  x\right)  $ the
semiconvex solution. The difficult nature of the fully nonlinear equation
(\ref{Esigma2}) is that its linearized operator matrix $\left(  F_{ij}\right)
$ is not uniformly elliptic; see (\ref{F_ij}) and (\ref{IEDfrange}). What
saves us is that the PDE for the Legendre-Lewy transform of $u(x)$ is
uniformly elliptic, found in [CY]. By the transformation rule Proposition
\ref{prop:trans}, the subharmonic $b$ in original variables corresponds to a
subharmonic $b$ in new variables for the linearized operator of the new,
uniformly elliptic equation. In new variables, the local maximum principle
implies a mean value inequality for the subharmonic $b,$ which upon pulling
back to original variables yields the mean value inequality in Proposition
\ref{prop:mvi}. The Hessian estimate becomes possible in Section
\ref{sec:Hess}. The Jacobi inequality combined with the divergence structure
of $F_{ij}$ allows us to bound the integral in terms of $\Vert Du\Vert
_{L^{\infty}}.$

\medskip The Hessian estimates for general solutions $K=\infty$ to quadratic
Hessian equation $\sigma_{2}\left(  D^{2}u\right)  =1$ in higher dimension
$n\geq4$ still remain an issue to us.

\section{Preliminaries}

\label{sec:Prelims}

Taking the gradient of both sides of the quadratic Hessian equation
(\ref{Esigma2}), we have%
\begin{equation}
\bigtriangleup_{F}Du=0, \label{Egradient}%
\end{equation}
where the linearized operator is given by
\begin{equation}
\bigtriangleup_{F}=\sum_{i,j=1}^{n}F_{ij}\partial_{ij}=\sum_{i,j=1}%
^{n}\partial_{i}\left(  F_{ij}\partial_{j}\right)  , \ \ \label{lin}%
\end{equation}
with%
\begin{equation}
\left(  F_{ij}\right)  =\bigtriangleup u\ I-D^{2}u=\sqrt{2+\left\vert
D^{2}u\right\vert ^{2}}\ I-D^{2}u>0. \label{F_ij}%
\end{equation}
Here without loss of generality, we assume $\bigtriangleup u>0.$ Otherwise the
smooth Hessian $D^{2}u$ would be in the $\bigtriangleup u<0$ branch of the
equation \eqref{Esigma2}. Given the semiconvexity condition, the conclusion in
Theorem \ref{thm:Hess} would be trivially true.

In passing, we add a quick proof of the quantitative ellipticity for equation
(\ref{Esigma2}) (again on the positive branch):%
\begin{align}
\frac{2}{\left(  n+1\right)  \lambda_{1}}  &  \leq F_{\lambda_{1}}\leq\left(
n-1\right)  \lambda_{1},\label{QEllip}\\
\left(  \sqrt{2}-1\right)  \lambda_{1}  &  \leq F_{\lambda_{i}}\leq\left(
n-1\right)  \lambda_{1}\ \ \text{for }i\geq2,\nonumber
\end{align}
which was first proved by Lin-Trudinger \cite{LT}, under the convention
$\lambda_{1}\geq\lambda_{2}\geq\cdots\geq\lambda_{n}.$ The upper bound is
straightforward. For the lower bound,%
\[
D_{1}\sigma_{2}=\sqrt{\left\vert \lambda\right\vert ^{2}+2}-\lambda_{1}%
=\frac{\left\vert \lambda^{\prime}\right\vert ^{2}+2}{\sqrt{\left\vert
\lambda\right\vert ^{2}+2}+\lambda_{1}}\geq\frac{2}{\sigma_{1}+\lambda_{1}%
}\geq\frac{2}{\left(  n+1\right)  \lambda_{1}};
\]
and when $i\geq2$,%
\[
D_{i}\sigma_{2}=\sqrt{\left\vert \lambda\right\vert ^{2}+2}-\lambda_{i}%
>\sqrt{\lambda_{1}^{2}+\lambda_{i}^{2}}-\lambda_{i}\geq\left(  \sqrt
{2}-1\right)  \lambda_{1},
\]
since function $\sqrt{\lambda_{1}^{2}+\lambda_{i}^{2}}-\lambda_{i}$ is
decreasing in terms of $\lambda_{i}.$

The gradient square $\left\vert \nabla_{F}v\right\vert ^{2}$ for any smooth
function $v$ with respect to the inverse \textquotedblleft
metric\textquotedblright\ $\left(  F_{ij}\right)  $ is defined as%
\[
\left\vert \nabla_{F}v\right\vert ^{2}=\sum_{i,j=1}^{n}F_{ij}\partial
_{i}v\partial_{j}v.
\]

\subsection{Jacobi inequality}

\label{sec:Jacobi}

Our objective in this subsection is to get a \textit{quantitative} subsolution
inequality for the maximum eigenvalue.

\begin{prop}
\label{prop:Jac} Let $u$ be a smooth solution to (\ref{Esigma2}) $\sigma
_{2}\left(  \lambda\right)  =1.$ Suppose that $\lambda_{1}>\lambda_{2}%
\geq\cdots\geq\lambda_{n}\geq-K$ and $\lambda_{1}\geq\Lambda\left(
n,K\right)  $ for some sufficiently large $\Lambda\left(  n,K\right)  $ at
$x=p.$ Set $b=\ln\lambda_{1}.$ Then we have at $p$%
\begin{equation}
\bigtriangleup_{F}b\geq\varepsilon\left\vert \nabla_{F}b\right\vert ^{2}
\label{Jacp}%
\end{equation}
for $\varepsilon=1/4,$ say.
\end{prop}

\begin{proof}
\textit{Step 1. Differentiation of maximum eigenvalue.}

We derive the following formulas for smooth function $b=\ln\lambda_{1}$
\begin{equation}
\left\vert \nabla_{F}b\right\vert ^{2}=\left(  b^{\prime}\right)  ^{2}%
\sum_{k=1}^{n}f_{k}u_{11k}^{2} \label{gradientb}%
\end{equation}
and%
\begin{equation}
\bigtriangleup_{F}b=\left\{
\begin{array}
[c]{c}%
b^{\prime}\left[  2\sum_{i>j}-u_{ii1}u_{jj1}+\sum_{k>1}\frac{2f_{k}}%
{\lambda_{1}-\lambda_{k}}u_{kk1}^{2}\right]  +b^{\prime\prime}f_{1}u_{111}%
^{2}\ \ \ \text{(I)}\\
+\sum_{i>1}\left[  2b^{\prime}+\left(  \frac{2b^{\prime}}{\lambda_{1}%
-\lambda_{i}}+b^{\prime\prime}\right)  f_{i}\right]  u_{11i}^{2}%
\ \ \ \ \ \ \ \ \ \ \ \ \ \text{(II)}\\
+\sum_{i>j>1}2b^{\prime}\left(  1+\frac{f_{i}}{\lambda_{1}-\lambda_{j}}%
+\frac{f_{j}}{\lambda_{1}-\lambda_{i}}\right)  u_{ij1}^{2}\ \ \ \text{(III)}%
\end{array}
\right.  \label{lapb}%
\end{equation}
at $p,$ where $D^{2}u$ is assumed to be diagonalized and $f\left(
\lambda\right)  =\sigma_{2}\left(  \lambda\right)  .$

To this end, we start with the partial derivatives of the distinct eigenvalue
$\lambda_{1}$ with respect to arbitrary unit vector $e\in\mathbb{R}^{n}$ at
$p$%
\begin{align*}
\partial_{e}\lambda_{1}  &  =\partial_{e}u_{11},\\
\partial_{ee}\lambda_{1}  &  =\partial_{ee}u_{11}+\sum_{k>1}2\frac{\left(
\partial_{e}u_{1k}\right)  ^{2}}{\lambda_{1}-\lambda_{k}},
\end{align*}
which can be reached for example by implicitly differentiating the
characteristic equation
\[
\det(D^{2}u-\lambda_{1}I)=0
\]
near any point where $\lambda_{1}$ is distinct from the other eigenvalues.

Thus we get (\ref{gradientb}) at $p$%
\[
\left\vert \nabla_{F}b\right\vert ^{2}=\sum_{k=1}^{n}F_{kk}\left(  b^{\prime
}\right)  ^{2}u_{11k}^{2}=\left(  b^{\prime}\right)  ^{2}\sum_{k=1}^{n}%
f_{k}u_{11k}^{2}.
\]

From
\[
\partial_{ee}b=b^{\prime}\partial_{ee}\lambda_{1}+b^{\prime\prime}\left(
\partial_{e}\lambda_{1}\right)  ^{2},
\]
we conclude that at $p$
\[
\partial_{ee}b=b^{\prime}\left[  \partial_{ee}u_{11}+\sum_{k>1}2\frac{\left(
\partial_{e}u_{1k}\right)  ^{2}}{\lambda_{1}-\lambda_{k}}\right]
+b^{\prime\prime}\left(  \partial_{e}u_{11}\right)  ^{2},
\]
and
\begin{align}
\bigtriangleup_{F}b  &  =\sum_{\gamma=1}^{n}F_{\gamma\gamma}\partial
_{\gamma\gamma}b\nonumber\\
&  =\sum_{\gamma=1}^{n}F_{\gamma\gamma}b^{\prime}\left(  \partial
_{\gamma\gamma}u_{11}+\sum_{k>1}2\frac{\left(  u_{1k\gamma}\right)  ^{2}%
}{\lambda_{1}-\lambda_{k}}\right)  +\sum_{\gamma=1}^{n}F_{\gamma\gamma
}b^{\prime\prime}u_{11\gamma}^{2}. \label{E4thorder}%
\end{align}

Next we substitute the fourth order derivative terms $\partial_{\gamma\gamma
}u_{11}$ in the above by lower order derivative terms. Differentiating
equation (\ref{Egradient}) $\sum_{\alpha,\beta=1}^{n}F_{\alpha\beta}%
u_{j\alpha\beta}=0$ and using (\ref{F_ij}),\ we obtain
\begin{align*}
\bigtriangleup_{F}u_{ij}  &  =\sum_{\alpha,\beta=1}^{n}F_{\alpha\beta
}u_{ji\alpha\beta}=\sum_{\alpha,\beta=1}^{n}-\partial_{i}F_{\alpha\beta
}u_{j\alpha\beta}=\sum_{\alpha,\beta=1}^{n}-\left(  \bigtriangleup
u_{i}\ \delta_{\alpha\beta}-u_{i\alpha\beta}\right)  u_{j\alpha\beta}\\
&  =\sum_{\alpha=1}^{n}-\left(  \bigtriangleup u_{i}-u_{i\alpha\alpha}\right)
u_{j\alpha\alpha}+\sum_{\alpha\neq\beta}u_{i\alpha\beta}u_{j\alpha\beta}%
=\sum_{\alpha\neq\beta}\left(  u_{i\alpha\beta}u_{j\alpha\beta}-u_{i\beta
\beta}u_{j\alpha\alpha}\right)  .
\end{align*}
Plugging the above identity with $i=j=1$ in (\ref{E4thorder}), we have at $p$
\[
\bigtriangleup_{F}b=b^{\prime}\left[  \sum_{i\neq j}\left(  u_{ij1}%
^{2}-u_{ii1}u_{jj1}\right)  +\sum_{\gamma=1}^{n}\sum_{k>1}2F_{\gamma\gamma
}\frac{u_{1k\gamma}^{2}}{\lambda_{1}-\lambda_{k}}\right]  +\sum_{\gamma=1}%
^{n}b^{\prime\prime}F_{\gamma\gamma}u_{11\gamma}^{2}.
\]
Regrouping those terms $u_{\heartsuit\heartsuit1}$ (with $u_{111}$),
$u_{11\heartsuit},$ and $u_{\heartsuit\clubsuit1}$ in the last expression,
noting $F_{\gamma\gamma}=f_{\gamma}$ at $p,$ we obtain (\ref{lapb}).

\textit{Step 2. Convexity of the level set of the equation }$\left\{  \left.
M\right\vert F\left(  M\right)  =0\right\}  .$

We rewrite the cross terms $2\sum_{i>j}-u_{ii1}u_{jj1}=2D^{2}F\left(
D^{2}u_{1},D^{2}u_{1}\right)  $ inside (I) of (\ref{lapb}) in a
\textquotedblleft positive\textquotedblright\ way
\begin{equation}
2\sum_{i>j}-u_{ii1}u_{jj1}=\sum_{i\neq j}-t_{i}t_{j}=\frac{\left(  \left\vert
\lambda\right\vert ^{2}+2\right)  \left\vert t\right\vert ^{2}-\left\langle
\lambda,t\right\rangle ^{2}}{\sigma_{1}^{2}}, \label{positive balance}%
\end{equation}
where we denoted $t_{i}=u_{ii1}.$ In fact, squaring the equation
(\ref{Egradient}) $\sum_{i=1}^{n}f_{i}t_{i}=0$ at $p,$ or equivalently
\[
\sigma_{1}\ \left(  t_{1}+\cdots+t_{n}\right)  =\lambda_{1}t_{1}%
+\cdots+\lambda_{n}t_{n},
\]
we have%
\[
\sigma_{1}^{2}\left(  \left\vert t\right\vert ^{2}+\sum_{i\neq j}t_{i}%
t_{j}\right)  =\left\langle \lambda,t\right\rangle ^{2}.
\]
Hence, the above \textquotedblleft positive\textquotedblright\ way follows
from equation (\ref{Esigma2}) $\sigma_{1}^{2}=\left\vert \lambda\right\vert
^{2}+2.$

\textit{Step 3. Consequence of semiconvexity }$\lambda_{i}\geq-K.$

We are ready to prove the Jacobi inequality (\ref{Jacp}). Note that all the
\textquotedblleft off-diagonal\textquotedblright\ terms in (III) of
(\ref{lapb}) are nonnegative; it follows that%
\[
\bigtriangleup_{F}b-\varepsilon\left\vert \nabla_{F}b\right\vert ^{2}%
\geq(I)-\varepsilon\left(  b^{\prime}\right)  ^{2}f_{1}t_{1}^{2}%
+(II)-\varepsilon\left(  b^{\prime}\right)  ^{2}\sum_{i>1}f_{i}u_{11i}^{2}.
\]

Now%
\begin{gather*}
(II)-\varepsilon\left(  b^{\prime}\right)  ^{2}\sum_{i>1}f_{i}u_{11i}%
^{2}=\left(  b^{\prime}\right)  ^{2}\sum_{i>1}\left[  2\lambda_{1}+\left(
\frac{2\lambda_{1}}{\lambda_{1}-\lambda_{i}}-1-\varepsilon\right)
f_{i}\right]  u_{11i}^{2}\\
\geq\left(  b^{\prime}\right)  ^{2}\sum_{i>1}\underset{\lambda_{i}%
\geq-K}{\underbrace{\left(  \frac{2\lambda_{1}}{\lambda_{1}+K}-1-\varepsilon
\right)  }}\ \underset{\text{(\ref{F_ij})}}{\underbrace{\left(  \sqrt
{\left\vert \lambda\right\vert ^{2}+2}-\lambda_{i}\right)  }}u_{11i}^{2}\geq0
\end{gather*}
for
\[
\lambda_{1}\geq\frac{1+\varepsilon}{1-\varepsilon}K\ \ \text{with }%
\varepsilon<1.
\]

Plugging (\ref{positive balance}) in $(I)$ of (\ref{lapb}), we have%
\[
\left(  I\right)  -\varepsilon b^{\prime2}f_{1}t_{1}^{2}=\frac{1}{\lambda
_{1}\sigma_{1}^{2}}\left\{
\begin{array}
[c]{c}%
\left(  \left\vert \lambda\right\vert ^{2}+2\right)  \left\vert t\right\vert
^{2}-\left\langle \lambda,t\right\rangle ^{2}+\sigma_{1}^{2}\sum_{k>1}%
\frac{2f_{k}}{\lambda_{1}-\lambda_{k}}\mathbf{t}_{k}^{2}\\
-\left(  1+\varepsilon\right)  \frac{\sigma_{1}^{2}}{\lambda_{1}^{2}%
}\underset{1+0.5\lambda^{\prime2}}{\underbrace{\lambda_{1}f_{1}}}t_{1}^{2}%
\end{array}
\right\}  .
\]
Observe that
\begin{equation}
f_{1}=\sigma_{1}-\lambda_{1}=\sqrt{\left\vert \lambda\right\vert ^{2}%
+2}-\lambda_{1}=\frac{\left\vert \lambda^{\prime}\right\vert ^{2}+2}%
{\sqrt{\left\vert \lambda\right\vert ^{2}+2}+\lambda_{1}}<\frac{0.5\left\vert
\lambda^{\prime}\right\vert ^{2}+1}{\lambda_{1}}, \label{ID1sigma2}%
\end{equation}
where $\lambda^{\prime}=\left(  \lambda_{2},\cdots,\lambda_{n}\right)  .$ Also
we see that
\begin{equation}
\left\vert \lambda_{k}\right\vert \leq C\left(  n,K\right)  \ \ \text{for\ \ }%
k\geq2 \label{IEallbutonebounded}%
\end{equation}
from (\ref{ID1sigma2}) (cf. [CY, p.663]). Indeed by the assumption
$\lambda_{n}\geq-K$ and $\left\vert \lambda^{\prime}\right\vert \leq
n\lambda_{1},$ we have%
\begin{align*}
-\left(  n-2\right)  K+\left\vert \lambda_{+}^{\prime}\right\vert  &
\leq\lambda_{n}+\cdots+\lambda_{2}=\frac{\left\vert \lambda^{\prime
}\right\vert ^{2}+2}{\sqrt{\left\vert \lambda\right\vert ^{2}+2}+\lambda_{1}%
}\\
&  <2+\frac{\left\vert \lambda^{\prime}\right\vert ^{2}}{\left(  1+1/n\right)
\left\vert \lambda^{\prime}\right\vert }\leq2+\left\vert \lambda_{-}^{\prime
}\right\vert +\frac{\left\vert \lambda_{+}^{\prime}\right\vert }{\left(
1+1/n\right)  },
\end{align*}
where $\lambda_{+}=\left(  \lambda_{2},\cdots,\lambda_{m}\right)  $ and
$\lambda_{-}=\left(  \lambda_{m+1},\cdots,\lambda_{n}\right)  $ for
$\lambda_{2}\geq\cdots\geq\lambda_{m}\geq0\geq\lambda_{m+1}\geq\cdots
\geq\lambda_{n}.$ Solving the above inequality for $\left\vert \lambda
_{+}\right\vert ,$ we get
\[
\left\vert \lambda_{+}\right\vert <\left(  n+1\right)  \left[  2+2\left(
n-2\right)  K\right]  =C\left(  n,K\right)  .
\]
Consequently, $\lambda_{1}\left(  x\right)  $ is a distinct eigenvalue, thus
smooth near $x=p$ if
\begin{equation}
\lambda_{1}\left(  p\right)  >C_{\text{smooth}}\left(  n,K\right)  ;
\label{smoothcut}%
\end{equation}%
\begin{equation}
c\left(  n,K\right)  \leq\lambda_{1}f_{1}\leq C\left(  n,K\right)
\ \ \text{and for }k\geq2,\ c\left(  n,K\right)  \leq\frac{f_{k}}{\lambda_{1}%
}\leq C\left(  n,K\right)  \ ; \label{IEDfrange}%
\end{equation}
and also%
\[
\frac{\sigma_{1}^{2}}{\lambda_{1}^{2}}=1+o\left(  1\right)  \ \ \text{and
}\sigma_{1}^{2}\frac{2f_{k}}{\lambda_{1}-\lambda_{k}}=\left[  2+o\left(
1\right)  \right]  \lambda_{1}^{2}\
\]
for large enough $\lambda_{1}$ and $k\geq2.\ $Denoting $t^{\prime}=\left(
t_{2},\dots,t_{n}\right)  .$ It follows that%
\begin{gather*}
\hspace*{-0.8in}\lambda_{1}\sigma_{1}^{2}\left[  \left(  I\right)
-\varepsilon b^{\prime2}f_{1}t_{1}^{2}\right]  \geq\\
\hspace*{-0.8in}\left.
\begin{array}
[c]{c}%
\left\{  \underline{\underline{\lambda_{1}^{2}}}+\left[  1-\varepsilon
-o\left(  1\right)  \right]  \left(  1+0.5\left\vert \lambda^{\prime
}\right\vert ^{2}\right)  \right\}  t_{1}^{2}+\left\{  \left[  \mathbf{3}%
+o\left(  1\right)  \right]  \lambda_{1}^{2}+\underline{\left\vert
\lambda^{\prime}\right\vert ^{2}}+2\right\}  \left\vert t^{\prime}\right\vert
^{2}\\
\\
-\underline{\underline{\lambda_{1}^{2}t_{1}^{2}}}-\underline{\left\vert
\lambda^{\prime}\right\vert ^{2}\left\vert t^{\prime}\right\vert ^{2}%
}-2\underset{\text{redistribute}}{\underbrace{t_{1}\left\vert \lambda^{\prime
}\right\vert \lambda_{1}\left\vert t^{\prime}\right\vert }}%
\end{array}
\right. \\
\geq\left.
\begin{array}
[c]{c}%
\left\{  \left[  1-\varepsilon-o\left(  1\right)  \right]  \left(
1+0.5\left\vert \lambda^{\prime}\right\vert ^{2}\right)  \right\}  t_{1}%
^{2}+\left[  \left(  \mathbf{3}+o\left(  1\right)  \right)  \lambda_{1}%
^{2}+2\right]  \left\vert t^{\prime}\right\vert ^{2}\\
-\left[  \left(  1-\varepsilon-o\left(  1\right)  \right)  \left(
1+0.5\left\vert \lambda^{\prime}\right\vert ^{2}\right)  \right]  t_{1}%
^{2}-\frac{\left\vert \lambda^{\prime}\right\vert ^{2}}{\left(  1-\varepsilon
-o\left(  1\right)  \right)  \left(  1+0.5\left\vert \lambda^{\prime
}\right\vert ^{2}\right)  }\lambda_{1}^{2}\left\vert t^{\prime}\right\vert
^{2}%
\end{array}
\right. \\
\geq\left\{  \left[  \mathbf{3}+o\left(  1\right)  \right]  -\frac
{2}{1-\varepsilon-o\left(  1\right)  }\right\}  \lambda_{1}^{2}\left\vert
t^{\prime}\right\vert ^{2}\geq0
\end{gather*}
for $\varepsilon<1/3,$ say $\varepsilon=1/4$ for simple notation and large
enough smooth
\begin{equation}
\lambda_{1}\geq\Lambda\left(  n,K\right)  >1+C_{\text{smooth}}\left(
n,K\right)  \label{Lambdacut}%
\end{equation}
with $C_{\text{smooth}}\left(  n,K\right)  $ from (\ref{smoothcut}).

We have proved the pointwise Jacobi inequality (\ref{Jacp}) in Proposition 2.1.
\end{proof}

\subsection{Integral Jacobi inequality}

Eventually in the proof of our Theorem 1.1, we use the following integral form
of (\ref{Jacp}).

From now on, repeated indices represent summation, unless otherwise indicated.

\begin{prop}
Let $u$ be a smooth $K$-convex (namely, $D^{2}u\geq-KI$) solution to $F\left(
D^{2}u\right)  =\sigma_{2}=1$ on $B_{3}$, and define the Lipschitz quantity
\[
b=\frac{1}{4}\ln\max(\Lambda,\lambda_{max}),
\]
where the sufficiently large $\Lambda=\Lambda(n,K)$ is from (\ref{Lambdacut}).
Then for all nonnegative $\varphi\in C_{c}^{\infty}(B_{3})$, there holds the
inequality
\begin{equation}
0\geq\int_{B_{3}}F_{ij}\varphi_{i}b_{j}+\varphi F_{ij}b_{i}b_{j}\,dx.
\label{IJac}%
\end{equation}

\end{prop}

\begin{proof}
It is easy to see that the Lipschitz function $b\left(  x\right)  $ is smooth
away from the level set $\left\{  \left.  x\right\vert \lambda_{\max}\left(
x\right)  =\Lambda\right\}  .$ By Sard's theorem, we perturb $\Lambda$ a tiny
bit, still denoted by $\Lambda,$ so that the Lipschitz $b\left(  x\right)  $
is smooth away from a zero measure set. Integrating by parts the pointwise
Jacobi inequality (\ref{Jacp}) multiplied by $\varphi,$ over a family of
approximated domains of $B_{3}$ from the complement of the above zero measure
set, we reach the integral Jacobi inequality (\ref{IJac}).
\end{proof}

\subsection{Legendre-Lewy transform}

\label{sec:CoV} In the integral approach of \cite{WY} toward Hessian estimates
for (\ref{Esigma2}) with $n=3$, a mean value inequality, pertaining to the
area structure on the Lagrangian minimal surface $(x,Du(x))\in\mathbb{R}%
^{2n},$ is used to bound $b(x)$ at $x=0$ by its integral. However, for $n>3,$
an area-like structure is unclear to us.

To construct a mean value inequality for subsolution $b$, in principle, we
would apply the local maximum principle, but the ellipticity constants for the
linearized operator of $\sigma_{2}=1$ are not uniform. To circumvent this, we
show that $b$ is a subsolution of a new uniformly elliptic operator after a
change of variables, which we describe below.

The $K$-convexity of $u$ ensures that the smallest canonical angle of the
\textquotedblleft Lewy-sheared\textquotedblright\ \textquotedblleft
gradient\textquotedblright\ graph $(x,Du(x)+Kx)$ is uniformly positive, i.e.
$\theta_{min}:=\arctan(\lambda_{min}+K)>0$. This means we can make a well
defined Legendre reflection about the origin,
\begin{equation}
(x,Du(x)+Kx)=(Dw(y),y), \label{LL}%
\end{equation}
where $w(y)$ is the Legendre transform of $u+\frac{K}{2}|x|^{2}$. Note that
$y(x)=Du(x)+Kx$ is a diffeomorphism.

\medskip We show here that this transformation preserves the linearized
operator of any fully nonlinear PDE, \textit{not just} $F\left(
D^{2}u\right)  =\sigma_{2}.$ Geometrically, this is clear for $K=0$ and the
special Lagrangian equation $\sum_{i=1}^{n}\arctan(\lambda_{i})=\Theta$, since
at the level of \textquotedblleft gradient\textquotedblright\ graphs, the
transformation is just a reflection, or a $\pi/2$-$U\left(  n\right)  $
rotation followed by a conjugation, so it only changes the constant phase
$\Theta.$

\begin{prop}
[Transformation rule]\label{prop:trans} Let $u$ solve $F(D^{2}u(x))=0$, and
its Legendre-Lewy transform $w(y)$ solve $G(D^{2}w(y)=-F(-KI+D^{2}%
w(y)^{-1})=0$. Then $L_{F}\approx L_{G}$, in the sense that for all smooth
functions $\varphi$, we have
\begin{equation}
\frac{\partial F}{\partial M_{ij}}(D^{2}u(x))\frac{\partial^{2}\varphi
}{\partial x^{i}\partial x^{j}}(x)=\frac{\partial G}{\partial N_{ij}}%
(D^{2}w(y))\frac{\partial^{2}\varphi^{\ast}}{\partial y^{i}\partial y^{j}}(y),
\label{invarianceofL}%
\end{equation}
where $\varphi^{\ast}(y)=\varphi(x(y))$.
\end{prop}

Equivalently, the right hand side will not have first order terms
$\partial\varphi(x(y))/\partial x^{i}$.

\begin{proof}
We will transform the left hand side of (\ref{invarianceofL}) into its right.
First,
\[
\frac{\partial\varphi}{\partial x^{j}}=\frac{\partial y^{k}}{\partial x^{j}%
}\frac{\partial\varphi^{\ast}}{\partial y^{k}}=(K\delta_{jk}+u_{jk}%
)\varphi_{k}^{\ast}.
\]
Consequently,
\[
F_{ij}\partial_{i}(\partial_{j}\varphi)=F_{ij}u_{ijk}\varphi_{k}^{\ast}%
+F_{ij}(K\delta_{jk}+u_{jk})\varphi_{k\ell}^{\ast}(K\delta_{i\ell}+u_{i\ell
}).
\]
The first term on the right hand side vanishes via the equation:
\[
F_{ij}u_{ijk}\varphi_{k}^{\ast}=\varphi_{k}^{\ast}\frac{\partial}{\partial
x^{k}}F(D^{2}u(x))=0.
\]
So it remains to verify that
\begin{equation}
(K\delta_{i\ell}+u_{i\ell})\frac{\partial F}{\partial M_{ij}}(D^{2}%
u(x))(K\delta_{jk}+u_{jk})=\frac{\partial G}{\partial N_{ij}}(D^{2}w(y)),
\label{F=G}%
\end{equation}
which is a little clearer in the eigenvalue dependent case
\[
F(M)=f(\lambda(M)),G(N)=g(\mu(N))=-f(-K+1/\mu(N)),
\]
since if the Hessian $D^{2}u(p)$ is diagonal at $x=p$ , then $K\delta_{i\ell
}+u_{i\ell}=(K+\lambda_{i})\delta_{i\ell}$, so that at $p$, the putative
equality is
\begin{equation}
\label{gi}(K+\lambda_{i})^{2}f_{i}=g_{i}.
\end{equation}
Since $(\lambda_{i}+K)^{2}f_{i}=(1/\mu_{i})^{2}\partial f/\partial\lambda
_{i}=\partial g/\partial\mu_{i}$, the result follows in this case.

Let us now return to the general situation. Using the chain rule for
$F(M-KI)=-G(M^{-1})$, we get
\begin{align*}
F_{ij}(M-KI)  &  =\frac{\partial}{\partial M_{ij}}(-G(M^{-1}))\\
&  =-\left.  \frac{\partial G}{\partial N_{k\ell}}\right\vert _{M^{-1}}%
\frac{\partial(M^{-1})_{k\ell}}{\partial M_{ij}}\\
&  =\left.  \frac{\partial G}{\partial N_{k\ell}}\right\vert _{M^{-1}}%
(M^{-1})_{ki}(M^{-1})_{\ell j},
\end{align*}
so upon multiplying by $K\delta_{i\ell}+u_{i\ell}=M_{i\ell}$, we obtain
(\ref{F=G}); in turn, the equivariance (\ref{invarianceofL}).
\end{proof}

\noindent\textbf{Remark 2.1. }The disappearance of gradient terms
$D\varphi^{\ast}\left(  y\right)  $ in the right hand side of
(\ref{invarianceofL}) depends on $u$ solving $F(D^{2}u)=0.$ For comparison,
without any equation for function $u\left(  x\right)  ,$ the Laplace-Beltrami
operator%
\[
\Delta_{g(x)}\varphi(x):=\frac{1}{\sqrt{\det g(x)}}\frac{\partial}{\partial
x^{i}}\left(  \sqrt{\det g(x)}\,g^{ij}(x)\frac{\partial}{\partial x^{j}%
}\varphi(x)\right)
\]
corresponding to the induced metric
\[
g(x)=\left.  dx^{2}+dy^{2}\right\vert _{L}=\left(  I+(D^{2}u(x))^{T}%
D^{2}u(x)\right)  dx^{2}%
\]
on the \textquotedblleft gradient\textquotedblright\ graph $L=\left(
x,Du(x)\right)  \in\mathbb{R}^{n}\times\mathbb{R}^{n}$ is invariant under any
rotation in $\mathbb{R}^{2n},$ in particular the Legendre transform,%
\[
\Delta_{g(x)}\varphi(x)=\Delta_{g(y)}\varphi^{\ast}(y).
\]
This is because, by design, the invariant Laplace-Beltrami operator%
\[
\Delta_{g}=g^{ij}\partial_{ij}+g^{ij}\Gamma_{ij}^{k}\partial_{k},
\]
carries over those first order derivative terms.

\bigskip

We now prove the mean value inequality using a transformation argument. We
suppose $Du(0)=0$ and $K$ is $K+1$ in transform \eqref{LL} for simplicity.

\begin{prop}
[Mean value inequality]\label{prop:mvi} Let $u$ be a smooth $K$-convex
solution to \eqref{Esigma2} on $B_{3}(0)$. If $b\in C(B_{3})$ is a viscosity
subsolution of the linearized operator \eqref{lin}, then the following
inequality holds:
\begin{equation}
\begin{aligned} b(0)\le C(n,K)\int_{B_1}b(x)\Delta u(x)\,dx. \end{aligned}
\end{equation}

\end{prop}

\begin{proof}
Let us first verify that transformed viscosity subsolution $b^{\ast
}(y):=b(x(y))$ is a viscosity subsolution of the transformed linearized
operator \eqref{invarianceofL}. We denote $y(B_{3}):=(K(\cdot)+Du)(B_{3})$ the
dual domain under the coordinate inversion. Suppose that $\psi(y)\in
C^{2}(y(B_{3}))$ touches $b^{\ast}(y)$ from above near $y_{0}\in y(B_{3}).$
Then $\psi_{\ast}(x):=\psi(y(x))$ touches $b(x)$ from above near
$x_{0}=x(y_{0})$, so
\[
F_{ij}\frac{\partial^{2}\psi_{\ast}}{\partial x^{i}\partial x^{j}}(x_{0}%
)\geq0.
\]
We recall $x\mapsto Kx+Du(x)$ is a diffeomorphism: letting $\varphi
(x)=\psi(y(x))\in C^{2}$, it follows from transformation rule
\eqref{invarianceofL} that $\varphi^{\ast}(y)=\psi(y)$ satisfies the desired
inequality at $y_{0}$.

\medskip It was first shown in \cite{CY} that the equation solved by the
vertical coordinate Lagrangian potential $w(y)$,
\[
G(D^{2}w)=-F(D^{2}u)=-\sigma_{2}(-KI+(D^{2}w)^{-1})=-1,
\]
is conformally, uniformly elliptic for $K$-convex solutions $u$, in the sense
that for $H_{ij}:=\sigma_{n}(\lambda(D^{2}w))G_{ij}$, the operator
$H_{ij}\partial_{ij}$ is uniformly elliptic:
\[
c(n,K)I\le(H_{ij})=\sigma_{n}(1/\lambda)(G_{ij})\le C(n,K)I.
\]
This can also be seen from \eqref{IEDfrange} using the change of variables \eqref{gi}.

\medskip Since $u\in C^{\infty}(B_{3})$ is $K$-convex, the gradient map
$y(x)=Du(x)+Kx$ is uniformly monotone, and we have a lower bound
$|y(x_{I})-y(x_{II})|\geq|x_{I}-x_{II}|$ for each $x_{I},x_{II}\in B_{3}$, if,
abusing notation for simplicity, $K$ is $K+1$ in our Legendre-Lewy transform
\eqref{LL}. It follows that the dual domain contains the unit ball:
\[
y(B_{1})=(K(\cdot)+Du)(B_{1}(0))\supset B_{1}(0).
\]
Since $b^{\ast}(y)$ is a subsolution of a uniformly elliptic operator, it
follows from the local maximum principle \cite[p.36]{CC} that $b^{\ast}$
satisfies a mean value inequality:
\[
b^{\ast}(0)\leq C(n,K)\int_{B_{1}^{y}}b^{\ast}(y)dy.
\]
Returning to $x$ variables and using $x(B_{1}^{y})\subset B_{1}$, we obtain
\[
b(0)\leq C(n,K)\int_{B_{1}}b(x)\det(D^{2}u(x)+KI)dx.
\]
Using $\lambda_{i}\leq C(n,K)$ with $i\geq2$ (for the small eigenvalues) from
(\ref{IEallbutonebounded}), as well as $c\left(  n\right)  \leq\lambda
_{max}=\lambda_{1}<\Delta u,$ we get
\[
b(0)\leq C(n,K)\int_{B_{1}}b(x)\lambda_{max}(x)\,dx\leq C(n,K)\int_{B_{1}%
}b(x)\Delta u(x)\,dx,
\]
as required.
\end{proof}

\noindent\textbf{Remark 2.2.} Without going through the Legendre-Lewy
transform, we do not see a direct proof for Proposition \ref{prop:mvi}. In the
original $x$-coordinates, in general, without the K-convexity assumption on
the solution $u\left(  x\right)  ,$ we have a weaker-quadratic-weight mean
value inequality than the one with the linear weight $\bigtriangleup u$ in
Proposition \ref{prop:mvi}. In fact, given any \textit{smooth} positive
subsolution $a\left(  x\right)  $, such as $\bigtriangleup u$, of linearized
operator \eqref{lin}, an easy modification of the local maximum principle
\cite[Theorem 9.20]{GT} yields the weighted mean value inequality
\[
a(0)\leq C(n)\int_{B_{1}}\left(  \frac{\Vert DF\Vert^{n}}{\det DF}\right)
a(x)dx,
\]
where $\Vert DF\Vert$ is the maximum eigenvalue of $(F_{ij})$. By the
eigenvalue bounds (\ref{QEllip}) of $(F_{ij})$, we have
\[
\frac{\Vert DF\Vert^{n}}{\det DF}\leq C(n)\lambda_{1}^{2}<C(n)(\Delta u)^{2},
\]
leading to the (ineffective!) mean value inequality
\[
a(0)\leq C(n)\int_{B_{1}}a(x)(\Delta u)^{2}dx.
\]
Still, there follows an $L^{\infty}$ Hessian bound for the solutions $u\left(
x\right)  $ to (\ref{Esigma2}) in terms of the $L^{3}$ norm of the Hessian
$D^{2}u$, improving a result in \cite{U2}.

\section{Proof of Theorem \ref{thm:Hess}}

\label{sec:Hess}

By scaling $v\left(  x\right)  =u\left(  Rx\right)  /R^{2},$ we assume $R=3$,
and we assume $Du(0)=0$ and $K$ is $K+1$ in \eqref{LL} for simplicity. By
Proposition \ref{prop:Jac}, $b(x)$ is a smooth subsolution of linearized
operator \eqref{lin} when $\lambda_{max}(x)\geq\Lambda(n,K)$ is sufficiently
large. Redefining it as
\[
b\left(  x\right)  =\max\left\{  \frac{1}{4}\ln\lambda_{\max}\left(  x\right)
,\frac{1}{4}\ln\Lambda\left(  n,K\right)  \right\}  ,
\]
we see that $b\left(  x\right)  ,$ as the maximum of two smooth subsolutions,
is a viscosity subsolution of linearized operator \eqref{lin}. By Proposition
\ref{prop:mvi}, we conclude it satisfies the mean value inequality
\[
b(0)\leq C(n,K)\int_{B_{1}}b(x)\Delta u(x)dx.
\]
Our next step is to apply the integral Jacobi inequality. Introducing a cutoff
function $\varphi=1$ on $B_{1}$ and $\varphi=0$ outside $B_{2}$, we integrate
by parts:
\begin{align*}
\int_{B_{1}}b(x)\Delta u(x)\,dx  &  \leq\int_{B_{2}}\varphi^{2}b(x)\Delta
u(x)\,dx\\
&  \leq-\int_{B_{2}}\varphi^{2}Db\cdot Du\ dx-2\int_{B_{2}}(\varphi
b)D\varphi\cdot Du\,dx.
\end{align*}
The second term is easy if we invoke $b\leq C(n,K)\ln\lambda_{max}\leq
C(n,K)\lambda_{max}\leq C(n,K)\Delta u$:
\[
-2\int_{B_{2}}(\varphi b)D\varphi\cdot Du\,dx\leq C(n,K)\Vert Du\Vert
_{L^{\infty}(B_{2})}\int_{B_{2}}\Delta u\,dx\leq C(n,K)\Vert Du\Vert
_{L^{\infty}(B_{2})}^{2}.
\]
For the first term, we start with
\[
-\int_{B_{2}}\varphi^{2}Db\cdot Du\,dx\leq C(n)\Vert Du\Vert_{L^{\infty}%
(B_{2})}\int_{B_{2}}|Db|\,dx.
\]
Next, the idea is to bound $\left\vert Db\right\vert $ by $F_{ij}b_{i}b_{j}.$
Assume at a point $p\in B_{2}$ that $D^{2}u(p)$ is diagonal, with
$u_{ii}=\lambda_{i}$ and $\lambda_{1}\geq\cdots\geq\lambda_{n}$. Write
$|Db(p)|\leq C(n)\sum_{i=1}^{n}|b_{i}(p)|$. For $i=1$:
\[
|b_{1}|\leq f_{1}b_{1}^{2}+1/f_{1}\leq f_{1}b_{1}^{2}+C(n,K)\lambda_{max},
\]
since $f_{1}\geq c(n,K)/\lambda_{1}$ from (\ref{IEDfrange}). For \textit{each
fixed} $i\geq2$:
\[
|b_{i}|\leq f_{i}b_{i}^{2}+1/f_{i}\leq f_{i}b_{i}^{2}+C(n,K),
\]
since $f_{i}\geq c(n,K)\lambda_{1}\geq c(n,K)>0$ from (\ref{IEDfrange}). We
conclude that in $B_{2}$,
\[
|Db|\leq C(n)F_{ij}b_{i}b_{j}+C(n,K)\Delta u,
\]
where we used $\Delta u\geq\sqrt{2n/\left(  n-1\right)  }.$ Therefore, we see
there is one term left to estimate:
\[
\int_{B_{2}}|Db|\,dx\leq C(n)\int_{B_{2}}F_{ij}b_{i}b_{j}\,\mathrm{d}%
x+C(n,K)\Vert Du\Vert_{L^{\infty}(B_{2})}.
\]
Let $\Phi$ be another cutoff, defined by $\Phi(x)=1$ on $B_{2}$, and $\Phi=0$
outside $B_{3}$. Applying the integral Jacobi inequality (\ref{IJac}), we can
write
\begin{align*}
\int_{B_{3}}\Phi^{2}F_{ij}b_{i}b_{j}\,dx  &  \leq-2\int_{B_{3}}F_{ij}\Phi
_{i}(\Phi b_{j})\,dx\\
&  \leq\frac{1}{2}\int_{B_{3}}\Phi^{2}F_{ij}b_{i}b_{j}\,dx+2\int_{B_{3}}%
F_{ij}\Phi_{i}\Phi_{j}\,dx.
\end{align*}
Thus, it remains to estimate the final term. Assume again that $D^{2}u(p)$ is
diagonal at $x=p$. Then at $p$, it is easy to estimate the integrand:
\[
F_{ij}\Phi_{i}\Phi_{j}=f_{i}\Phi_{i}^{2}\leq C(n,K)\sum_{i=1}^{n}%
f_{i}=C(n,K)(n-1)\Delta u.
\]
We conclude the final integral has the desired bound:
\[
\int_{B_{3}}F_{ij}\Phi_{i}\Phi_{j}\,dx\leq C(n,K)\Vert Du\Vert_{L^{\infty
}(B_{3})}.
\]
Putting all the pieces together, we conclude
\[
b(0)\leq C(n,K)\Vert Du\Vert_{L^{\infty}(B_{3})}^{2}.
\]
which completes the proof of Theorem \ref{thm:Hess}.

\end{document}